\documentclass [11pt,a4paper]{amsart}

\usepackage{amssymb} 
\usepackage{amsfonts} 
\usepackage{amsmath}
\usepackage{amsthm} 
\usepackage{tikz}
\usepackage{stmaryrd}
\usepackage[utf8]{inputenc}
\usepackage[all]{xy}
\SelectTips{cm}{}

\usepackage[pagebackref,
    ,pdfborder={0 0 0}
    ,urlcolor=black,a4paper,hypertexnames=false]{hyperref}
\hypersetup{pdfauthor={Clara L\"oh},pdftitle=Cost vs.\ integral 
foliated simplicial volume}

\usepackage{caption}

\DeclareMathOperator{\dom}{dom}

\DeclareMathOperator{\cost}{Cost}
\DeclareMathOperator{\rg}{rg}
\DeclareMathOperator{\lf}{lf}
\DeclareMathOperator{\supp}{supp}

\newtheorem{lem}{Lemma}[section]
\newtheorem{thm}[lem]{Theorem}

\newtheorem{cor}[lem]{Corollary} 

\theoremstyle{definition}
\newtheorem{defi}[lem]{Definition}

\newtheorem{question}[lem]{Question}

\newtheorem{exa}[lem]{Example}
\newtheorem{rem}[lem]{Remark} 

\newcommand{\N}{\ensuremath {\mathbb{N}}}
\newcommand{\F}{\ensuremath {\mathbb{F}}}
\newcommand{\R} {\ensuremath {\mathbb{R}}}

\newcommand{\Z} {\ensuremath {\mathbb{Z}}}

\DeclareMathOperator{\rk}{rk}

 \makeatletter
 \newcommand\norm{\bBigg@{0.8}}

 \makeatother
 \newcommand{\inparens}[2][flex]{\csname #1l\endcsname(#2%
                                 \csname #1r\endcsname)\mathclose{}}
 \newcommand{\inangles}[2][flex]{\csname #1l\endcsname\langle#2%
                                 \csname #1r\endcsname\rangle\mathclose{}} 
 \newcommand{\innorm}[2][flex]{\csname #1l\endcsname|#2%
                                 \csname #1r\endcsname|\mathclose{}}
 \newcommand{\indnorm}[2][flex]{\csname #1l\endcsname\|#2%
                                 \csname #1r\endcsname\|\mathclose{}}
 \newcommand{\indnorml}[4][flex]{\csname #1l\endcsname\|#2%
                                 \csname #1r\endcsname\|_{#3}^{#4}\mathclose{}}

\newcommand{\sv}[2][flex]{\indnorm[#1]{#2}}%
\newcommand{\isv}[2][norm]{\indnorml[#1]{#2}{\Z}{}}
\newcommand{\pfcl}[2][flex]{\csname #1l\endcsname[#2%
                            \csname #1r\endcsname]}
\newcommand{\ifsv}[2][norm]{\csname #1l\endcsname\bracevert\!#2\!%
                            \csname #1r\endcsname\bracevert}
\newcommand{\ifsvp}[3][norm]{\csname #1l\endcsname\bracevert\!#2\!%
                            \csname #1r\endcsname\bracevert\!^{#3}}
\newcommand{\stisv}[2][flex]{\indnorml[#1]{#2}{\Z}{\infty}}

\newcommand{\suv}[3][norm]{\csname #1l\endcsname\bracevert\!#2\!%
                           \csname #1r\endcsname\bracevert_{(#3)}}

\def\Clf#1#2#3{C_{#1}^{\lf}(#2;#3)}
\def\Hlf#1#2#3{H_{#1}^{\lf}(#2;#3)}

\def\fa#1{\forall_{#1}\quad}
\def\exi#1{\exists_{#1}\quad}

\newcommand{\pfc}[1]{%
  \widehat{#1}}

\newcommand{\linfz}[1]{%
  L^\infty(#1,\Z)}

\DeclareMathOperator{\map}{map}

\newcommand{\actson}{\curvearrowright}

\title{Cost vs.\ integral foliated simplicial volume}

\author{Clara L\"oh}



\def\draftinfo{}
\date{\today.\ \copyright{\ C.~L\"oh 2018}. 
    This work was supported by the CRC~1085 \emph{Higher Invariants} 
    (Universit\"at Regensburg, funded by the DFG).
    \draftinfo}

\begin{document}

\begin{abstract}
  We show that integral foliated simplicial volume of closed manifolds 
  gives an upper bound for the cost of the corresponding fundamental groups. 
\end{abstract}
\maketitle

\section{Introduction}

The dynamical view on groups and spaces aims at understanding
groups and topological spaces through actions on probability spaces. 
If $\Gamma$ is a group and $\alpha = \Gamma \actson (X,\mu)$ is a
measure preserving action on a probability space~$(X,\mu)$, then one considers,
for instance, the following invariants 
(see Section~\ref{sec:cost} and~\ref{sec:ifsv} for definitions and references):
The cost~$\cost_\mu \alpha$ of~$\alpha$ is a randomised version
of the minimal number of generators of~$\Gamma$. The cost~$\cost
\Gamma$ of the group~$\Gamma$ is the infimum of all such~$\cost_\mu \alpha$.

If $M$ is an oriented closed connected manifold with fundamental
group~$\Gamma$, then the $\alpha$-parametrised simplicial
volume~$\ifsvp M \alpha$ of~$M$ is a randomised version of the
integral simplicial volume of~$M$. The integral foliated simplicial
volume~$\ifsv M$ of~$M$ is the infimum of all such~$\ifsvp M \alpha$.

In the residually finite case, the profinite completion provides a
link between the dynamical view and the residually finite view:
If $\Gamma$ is a residually finite group, then the cost of the
translation action of~$\Gamma$ on the profinite completion~$\widehat
\Gamma$ coincides with the rank gradient~$\rg \Gamma$ of~$\Gamma$
(plus~$1$)~\cite[Theorem~1]{abertnikolov} and the corresponding
parametrised simplicial volume of~$M$ coincides with the stable
integral simplicial volume~$\stisv M$~\cite[Remark~6.7]{loehpagliantini}.
Moreover, these gradient invariants are related as follows: 

\begin{thm}[rank gradient estimate~\protect{\cite[Theorem~1.1]{loehrg}}]\label{resfin}
  If $M$ is an oriented closed connected manifold with fundamental
  group~$\Gamma$, then
  \[\rg \Gamma \leq \stisv M.
  \]
\end{thm}

It is natural to wonder whether the corresponding dynamical estimate
also holds~\cite[Question~1.3]{loehrg}. 
In the present article, we will complete the dynamical part of the picture 
by proving the following estimate:

\begin{thm}[cost estimate]\label{mainthm}
  Let $M$ be an oriented closed connected manifold with fundamental
  group~$\Gamma$ and let $\alpha = \Gamma\actson (X,\mu)$ be an essentially
  free ergodic standard $\Gamma$-space. Then
  \[ \cost_\mu \alpha - 1 \leq \ifsvp M \alpha.
  \]
  In particular,
  $\cost \Gamma - 1 \leq \ifsv M$. 
\end{thm}

Theorem~\ref{mainthm} shows that integral foliated simplicial volume
is a higher-dimensional, geometric refinement of the cost of
groups. In particular, as in the case of the rank gradient estimate,
the bound in Theorem~\ref{mainthm} is far from being sharp in
general: If $M$ is
an oriented closed connected hyperbolic surface, then~\cite{mschmidt,vbc}
\[ \ifsv {M \times M} \geq \sv {M \times M} \geq \sv M \cdot \sv M > 0, 
\]
but $\cost (\pi_1(M) \times \pi_1(M)) = 1$~\cite[Proposition~35.1]{kechrismiller}. 

The dependence of~$\cost_\mu \alpha$ and $\ifsvp M \alpha$ on the
chosen dynamical system~$\alpha$ is a delicate open
problem~\cite{gaboriaucost,kechrismiller}\cite[Section~1.5]{FLPS}.  In
analogy with the terminology for cost of groups, we define:

\begin{defi}[cheap manifold, manifold of fixed price]
  Let $M$ be an oriented closed connected manifold.
  \begin{itemize}
  \item The manifold~$M$ is \emph{cheap} if~$\ifsv M = 0$.
  \item The manifold~$M$ has \emph{fixed price} if for all essentially
    free standard $\pi_1(M)$-spaces~$\alpha$ and~$\beta$ we have
    $\ifsvp M \alpha = \ifsvp M \beta$.
  \end{itemize}
\end{defi}

As for groups, it is not known whether all manifolds have fixed price.

\begin{cor}
  Let $M$ be an oriented closed connected manifold with fundamental group~$\Gamma$.
  \begin{enumerate}
  \item If $M$ is cheap, then $\Gamma$ is cheap.
  \item If $M$ is cheap and of fixed price, then $\Gamma$ is cheap and of fixed price.
  \end{enumerate}
\end{cor}
\begin{proof}
  Let $M$ be cheap. We first observe that this implies that $\Gamma$
  is infinite (if $\Gamma$ is finite, then~$\ifsv M = 1 / |\Gamma|
  \cdot \isv {\widetilde M}$~\cite[Corollary~6.3]{loehpagliantini}, which is non-zero). 
  Because $\Gamma$ is infinite, $\cost \Gamma \geq 1$~\cite[p.~108]{kechrismiller}.  
  On the other hand, Theorem~\ref{mainthm} yields that
  $\cost \Gamma \leq \ifsv M + 1 = 1$; therefore, $\cost \Gamma = 1$, which means
  that $\Gamma$ is cheap.

  Let now $M$ additionally have fixed price.  In view of ergodic
  decomposition~\cite[Corollary~18.6]{kechrismiller}, it suffices to
  show that $\cost_\mu \alpha = 1$ holds for all essentially free
  \emph{ergodic} standard $\Gamma$-spaces~$\alpha = \Gamma \actson
  (X,\mu)$. In this case, again Theorem~\ref{mainthm} shows that
  \[ 1 \leq \cost \Gamma \leq \cost_\mu \alpha \leq \ifsvp M \alpha + 1 = 1.
  \qedhere
  \]
\end{proof}

The class of cheap manifolds of fixed price is known to include
all oriented closed connected manifolds that 
\begin{itemize}
\item are aspherical and have infinite amenable fundamental group~\cite[Theorem~1.9]{FLPS},
\item are smooth and admit a smooth $S^1$-action without fixed points
  and whose orbits are $\pi_1$-injective~\cite[Theorem~1.1]{fauser} or
  that are smooth and aspherical and admit a smooth non-trivial
  $S^1$-action~\cite[Corollary~1.2]{fauser},
\item are generalised graph manifolds~\cite[Theorem~1.6]{fauserfriedlloeh}, 
\item are a product of a cheap manifold of fixed price and another
  manifold~\cite[proof of Theorem~5.34]{mschmidt},
\item are smooth, aspherical, and have trivial minimal volume~\cite[(proof of) Corollary~5.4]{braun}.
\end{itemize}

For the sake of completeness, we put Theorem~\ref{mainthm} in context
with $L^2$-Bet\-ti numbers: For $L^2$-Betti numbers, we a have a
harmonic correspondence between the classical, the dynamical, and the
residually finite view: $L^2$-Bet\-ti numbers of compact manifolds can
be described both in the residually finite view (as Betti number
gradients)~\cite{lueckapprox} and in the dynamical view (as
$L^2$-Betti numbers of orbit relations)~\cite{gaboriaul2}.  In
contrast, it is known that stable integral simplicial volume and
integral foliated simplicial volume, in general do \emph{not} coincide
with the classicial simplicial volume of aspherical oriented closed
connected manifolds~\cite[Theorem~2.1]{FFM}\cite[Theorem~1.8]{FLPS}.
Integral foliated simplicial volumes and $L^2$-Betti numbers of
an oriented closed connected manifold~$M$ are for every~$k \in \N$ linked
by the following chain of inequalities~\cite[Corollary~5.28 (the
  constant factor can be improved
  to~$1$)]{mschmidt}\cite[Proposition~6.1]{loehpagliantini}:
\[ b_k^{(2)}(M) \leq \ifsv M \leq \stisv M.
\]
Moreover, it is known that $b_1^{(2)}(\pi_1(M)) \leq \cost \pi_1(M)
-1$ (if $\pi_1(M)$ is infinite)~\cite[Corollaire~3.23]{gaboriaul2}. Hence,
Theorem~\ref{mainthm} is a refinement of this chain of
inequalities in degree~$1$:
\[ b_1^{(2)}(M) \leq \cost \pi_1(M) - 1 \leq \ifsv M \leq \stisv M.
\]

However, the following problem  remains open:

\begin{question}\label{q:svcost}
  Let $M$ be an oriented closed connected aspherical manifold whose
  simplicial volume~$\sv M$ satisfies~$\sv M = 0$. 
  Does this already imply that $\pi_1(M)$ is cheap?
\end{question}

\begin{rem}\label{rem:triv}
  If the Singer conjecture for $L^2$-Betti numbers is true
  and the conjecture that $b^{(2)}_1(\Gamma) = \cost \Gamma -1$ holds for
  every (finitely presented infinite) group is true, then Question~\ref{q:svcost}
  clearly has a positive answer (even independently of the simplicial volume
  in dimension at least~$3$). However, as these two conjectures seem
  to be wild and wide open, it would be interesting to find an alternative, direct,
  answer to Question~\ref{q:svcost}.
\end{rem}

\subsection*{Organisation of this article}

We first review the notion of cost of standard equivalence relations
(Section~\ref{sec:cost}) and establish a basic estimate for cost of
certain subrelations (Section~\ref{subsec:transfin}). We then recall the
notion of integral foliated simplicial volume
(Section~\ref{sec:ifsv}).  In Section~\ref{sec:proof}, we will prove
Theorem~\ref{mainthm}. Finally, in Section~\ref{subsec:weightless}, we
will discuss the weightless version of Theorem~\ref{mainthm}.

\subsection*{Acknowledgements}

I would like to thank Daniel Fauser for many helpful discussions.

\section{Cost}\label{sec:cost}

The cost of a dynamical system of a group is a randomised version of
the rank (i.e., minimal number of generators) of the group.  More
generally, one can consider the cost of standard Borel equivalence relations on
measure spaces.  More information about these subjects can be found in
the literature~\cite{gaboriaucost,gaboriaul2,kechrismiller}.

\subsection{Standard actions and equivalence relations}

We will use the following notation and conventions on standard equivalence relations:

A \emph{standard Borel measure [probability] space} is a measure space 
[probability space]~$(X,\mu)$, where the measurable space~$X$ is isomorphic to
a Polish space with its Borel $\sigma$-algebra. For simplicity, we will only consider
the case of standard Borel measure spaces with finite total measure.

A \emph{measurable equivalence relation} on a standard Borel (measure)
space~$X$ is a measurable subset~$S \subset X \times X$ that is an
equivalence relation on~$X$. The \emph{automorphism group of~$S$} (or
\emph{full group of~$S$}) is the group~$[S]$ (via composition) of
measurable isomorphisms~$f\colon X \longrightarrow X$ 
that satisfy
\[ \fa{x,y\in X} x \sim_S y \Longrightarrow f(x) \sim_S f(y).
\]
Moreover, $\llbracket S\rrbracket$ denotes the set of partial automorphisms
of~$S$, i.e., of measurable isomorphisms~$f \colon A \longrightarrow B$
between measurable subsets~$A,B \subset X$ that satisfy
\[ \fa{x,y\in A} x \sim_S y \Longrightarrow f(x) \sim_S f(y);
\]
we write~$\dom f := A$ for the \emph{domain} of~$f$. 

\begin{defi}[standard equivalence relation]
  A \emph{standard equivalence relation} on a standard Borel measure
  space~$(X,\mu)$ is a measurable equivalence relation~$S$ on~$X$,
  where each equivalence class has cardinality at most~$|\N|$ and where
  each element of~$[S]$ is $\mu$-preserving.
\end{defi}

One of the key objects of measurable group theory and the dynamical
view is the orbit relation of a group action:

\begin{exa}[orbit relation of an action]
  Let $\Gamma$ be a group. A \emph{standard $\Gamma$-space} is a
  standard Borel probability space~$(X,\mu)$ together with a
  measurable $\mu$-preserving (left) action of~$\Gamma$ on~$(X,\mu)$.

  If $\Gamma$ is countable and $\alpha = \Gamma \actson (X,\mu)$ is a
  standard $\Gamma$-space, then the \emph{orbit relation}
  \[ \bigl\{ (x, \gamma \cdot x) \bigm| x \in X,\ \gamma \in \Gamma \bigr\} \subset X \times X
  \]
  is a standard equivalence relation in the sense above.

  Conversely, it can be shown that every standard equivalence relation arises
  as orbit relation of a suitable action of a suitable countable group on
  the underlying standard Borel measure space~\cite[Theorem~1]{feldmanmoore}.
\end{exa}

Moreover, we will need the following terms and constructions: Let $S$ be a
standard equivalence relation on a standard Borel measure space~$(X,\mu)$.
\begin{itemize}
\item If $A \subset X$ is a measurable subset with~$\mu(A) >0$, then the
  restriction~$\mu|_A$ of~$\mu$ to~$A$ turns~$(A,\mu|_A)$ into a standard
  Borel measure space.
\item In this situation, the restriction
  \[ S|_A := \bigl\{ (x,y) \in A \times A \bigm| x \sim_S y \bigr\} 
  \]
  of~$S$ to~$A$ is a standard equivalence relation.
\item A measurable subset~$A \subset X$ is a \emph{[almost] complete
  section of~$S$}, if for [$\mu$-almost] every~$x \in X$ there is a~$y
  \in A$ with~$x \sim_S y$.
\item The relation~$S$ on~$X$ is \emph{aperiodic}, if for $\mu$-almost
  every~$x \in X$ the $S$-orbit~$S \cdot x := \{ y \in X \mid y \sim_S x \}$
  is infinite.
\item A measurable subset~$A \subset X$ is \emph{$S$-invariant} if
  \[ \fa{x \in A} S\cdot x \subset A.
  \]
\item The relation~$S$ on~$(X,\mu)$ is \emph{ergodic}, if every $S$-invariant
  measurable subset~$A$ of~$X$ satisfies~$\mu(A) = 0$ or~$\mu (X \setminus A) = 0$.
\end{itemize}

\subsection{Cost of a standard equivalence relation}\label{subsec:costrel}

The cost of a standard equivalence relation is the minimal ``number'' of ``generators''
needed to describe the relation: 

\begin{defi}[graphing, cost]
  Let $S$ be a standard equivalence relation on a standard Borel measure space~$(X,\mu)$.
  \begin{itemize}
  \item Let $\Phi = (\varphi_i)_{i \in I}$ be a family of elements of~$\llbracket S\rrbracket$.
    Then
    \[ \langle \Phi \rangle_X
    := \Bigl\langle \bigcup_{i \in I} \bigl\{ (x, \varphi_i(x)) \in X \times X
                                    \mid x \in \dom \varphi_i \bigr\}
       \Bigr\rangle_X
    \]
    denotes the smallest (with respect to inclusion) equivalence relation on~$X$
    containing the given set of pairs.
  \item A \emph{graphing of~$S$} is a family~$\Phi = (\varphi_i)_{i \in I}$
    in~$\llbracket S \rrbracket$ with~$\langle \Phi\rangle_X = S$. The \emph{cost}
    of~$\Phi$ is defined as
    \[ \cost_\mu \Phi := \sum_{i \in I} \mu(\dom \varphi_i).
    \]
  \item The \emph{cost}~$\cost_\mu S$ of~$S$ is the infimum of all costs of graphings of~$S$. 
  \end{itemize}
\end{defi}

\begin{defi}[cost of a group~\cite{gaboriaucost}]
  The \emph{cost}~$\cost \Gamma$ of a countable group~$\Gamma$
  is the infimum of all costs of orbit relations of standard $\Gamma$-spaces.
\end{defi}

\begin{exa}
  Let $\Gamma$ be a finitely generated group. Then $\cost \Gamma \leq
  \rk \Gamma$, where $\rk \Gamma$ denotes the minimal number of
  generators of~$\Gamma$ (as witnessed by the translation
  automorphisms associated with a smallest generating set).  If
  $\Gamma$ in addition is residually finite and infinite, then the
  translation action of~$\Gamma$ on its profinite completion~$\pfc
  \Gamma$ is a standard $\Gamma$-space and~\cite[Theorem~1]{abertnikolov}
  \[ \cost_\mu (\Gamma \actson \pfc \Gamma) -1 = \rg \Gamma 
  \]
  (where $\rg \Gamma$ denotes the rank gradient of~$\Gamma$). 
\end{exa}

In all cases, where the cost~$\cost \Gamma$ of a countable infinite
group~$\Gamma$ could be computed so far, it coincides
with~$b^{(2)}_1(\Gamma) + 1$. This includes, for example, amenable
groups, free groups, etc.~\cite{gaboriaucost,gaboriaul2,kechrismiller}.

\subsection{Cost and restrictions}

We collect basic facts on cost with respect to restrictions.

\begin{lem}[cost of partitions~\protect{\cite[p.~60]{kechrismiller}}]\label{lem:partition}
  Let $(X,\mu)$ be a standard Borel measure space, let $R$ be a standard 
  equivalence relation on~$X$, and let $X = \bigcup_{j=1}^m A_j$ be a
  partition of~$X$ into measurable $R$-invariant subsets of non-zero measure. 
  Then
  \[ \cost_\mu R = \sum_{j=1}^m \cost_{\mu|_{A_j}} R|_{A_j}.
  \]
\end{lem}

\begin{lem}[cost of complete sections~\protect{\cite[Proposition~II.6]{gaboriaucost}\cite[Proposition~21.1]{kechrismiller}}]\label{lem:completesec}
  Let $(X,\mu)$ be a standard Borel measure space, let $R$ be a standard 
  equivalence relation on~$X$, and let $A \subset X$ be a complete section of~$R$.
  Then
  \[ \cost_\mu R = \cost_{\mu|_A}(R|_A) + \mu(X \setminus A).
  \]
\end{lem}

\begin{lem}[cost of restrictions]\label{lem:restrict}
  Let $(X,\mu)$ be a standard Borel probability space, let $R$ be a standard
  equivalence relation on~$X$, and let $A \subset X$ be a measurable subset
  with~$\mu(A) > 0$. Then
  \[ \cost_{\mu|_A} (R|_A) \leq \cost_\mu R.
  \]
\end{lem}
\begin{proof}
  Let $B := R \cdot A = \bigcup_{x \in A} R \cdot x \subset X$. Then
  $B$ is a measurable subset of~$X$~\cite[p.~291]{feldmanmoore}. By
  construction, $B$ and $X \setminus B$ are measurable $R$-invariant
  subsets of~$X$. If $\mu(X \setminus B) = 0$, then $A$ is an (almost)
  complete section of~$R$ and Lemma~\ref{lem:completesec} shows that
  \[ \cost_{\mu|_A}(R|_A) = \cost_\mu R - \mu(X \setminus A) \leq \cost_\mu R.
  \] 
  If $\mu(X\setminus B) \neq 0$, then we can apply Lemma~\ref{lem:partition}
  and Lemma~\ref{lem:completesec} (because $A$ is a complete section of~$R|_B$ on~$B$)
  to obtain
  \begin{align*}
    \cost_{\mu} R
    & = \cost_{\mu|_B} (R|_B) + \cost_{\mu|_{X \setminus B}} (R|_{X \setminus B})
    \\
    & \geq \cost_{\mu|_A} (R|_A) + \mu(B \setminus A)
    \\
    & \geq \cost_{\mu|_A} (R|_A).
  \qedhere 
  \end{align*}
\end{proof}

\subsection{Cost of translation finite extensions}\label{subsec:transfin}

The key estimate in the proof of Theorem~\ref{mainthm} will involve
the following variation of the notion of finite index subrelations.
In contrast with finite index subrelations, we only require that the
orbits of the ambient relation can be covered, in a uniform way, by
finitely many translates of orbits of the subrelation:

\begin{defi}[translation finite extension]\label{def:transfin}
  Let $(X,\mu)$ be a standard Borel probability space, let $S$ be a
  standard equivalence relation on~$(X,\mu)$, and let $R \subset S$ be
  a standard equivalence relation on~$(X,\mu)$ that is contained in~$S$. Then
  $R \subset S$ is a \emph{translation finite extension} if there
  exists a finite set~$F \subset [S]$ such that for $\mu$-almost every~$x \in X$
  we have
  \[ S \cdot x = \bigcup_{f,g \in F} f\bigl( R \cdot g^{-1}(x)\bigr).
  \]
\end{defi}

\begin{exa}
  Let $(X,\mu)$ be a standard Borel probability space.
  \begin{itemize}
  \item Let $\Gamma \actson (X,\mu)$ be a standard $\Gamma$-space and let $S$
    be the corresponding orbit relation on~$X$. Moreover, let $\Lambda \subset \Gamma$
    be a finite index subgroup and let $R \subset S$ be the orbit relation of the
    action restricted to~$\Lambda$. Then $R \subset S$ is a translation finite
    extension (witnessed by the left translations of a set of coset representatives). 

    In this case, $R$ is even a subrelation of finite index of~$S$ and thus 
    we have~$\cost_\mu S \leq \cost_\mu
    R$~\cite[Proposition~VI.23]{gaboriaucost}\cite[Proposition~25.1]{kechrismiller}.
  \item
    Conversely, if $R \subset S$ is a translation finite extension of standard 
    equivalence relations on~$(X,\mu)$, then
    $R$ does not necessarily have finite index in~$S$:
    We consider the circle~$S^1 = [0,1]/(0\sim 1)$ with the Lebesgue probability
    measure~$\mu$ and a $\Z$-action by irrational rotation. Let $S$ be the corresponding
    orbit relation. 
    Let $\pi \colon [0,1] \longrightarrow S^1$ be the canonical projection, let $A := \pi([0,1/2])$,
    and let
    \[ R := \langle S|_A \rangle_X \subset S.
    \]
    Then $R$ does \emph{not} have finite index in~$S$, but $R \subset S$
    is a translation finite extension. 

    Moreover, $\cost_\mu(S) = 1$~\cite[Corollaire~III.4]{gaboriaucost}\cite[Corollary~31.2]{kechrismiller}
    and (Lemma~\ref{lem:partition} and Lemma~\ref{lem:completesec})
    \begin{align*}
      \cost_\mu (R)
      & = \cost_{\mu|_A}(R|_A) + \cost_{\mu|_{S^1 \setminus A}} (R|_{S^1 \setminus A})
      \\
      & = \cost_{\mu|_A}(S|_A) + 0
      \\
      & = \cost_{\mu}(S) - \mu(S^1 \setminus A)
      \\
      & = 1 - \frac12 = \frac12.
    \end{align*}
    In particular, in this case we have~$\cost_\mu S \not\leq \cost_\mu R$.
  \end{itemize}
\end{exa}

\begin{lem}[cost estimate for translation-finite extensions]\label{lem:transfin}
  Let $(X,\mu)$ be a standard Borel probability space, let $S$ be an
  aperiodic ergodic standard relation on~$(X,\mu)$, and let $R \subset S$ be a
  translation finite extension. Then
  \[ \cost_\mu S \leq \cost_\mu R + 1.
  \]
\end{lem}
\begin{proof}
  The proof is a straightforward adaption of the (stronger) cost estimate
  for finite index subrelations~\cite[Proposition~25.1]{kechrismiller}: 
  Because $R \subset S$ is translation finite extension, there exists
  a finite set~$F \subset [S]$ such that for $\mu$-almost every~$x \in X$
  we have
  \[ S \cdot x = \bigcup_{f,g \in F} f\bigl( R \cdot g^{-1}(x)\bigr).
  \]
  Passing to an $S$-invariant co-null subset, we may assume without loss of
  generality that this even holds for every~$x \in X$. 
  Let $\varepsilon \in \R_{>0}$
  and let $\Phi$ be a graphing of~$R$ with
  \[ \cost_\mu(\Phi) \leq \cost_\mu R + \varepsilon.
  \]
  Furthermore, let $A \subset X$ be a complete Borel section of~$S$ 
  with~$0 < \mu(A) < \varepsilon$ (such a set does exist~\cite[Lemma~6.7]{kechrismiller}). 
  Thus, by Lemma~\ref{lem:completesec}, 
  \[ \cost_\mu(S) = \cost_\mu(S|_A) + \mu(X \setminus A)
     \leq \cost_{\mu|_A}(S|_A) + 1.
  \]

  For~$f\in F$, we let 
  \begin{align*}
    \varphi_f := f^{-1}|_A \colon A & \longrightarrow f^{-1}(A)
    \\
    x & \longmapsto f^{-1}(x)
  \end{align*}
  and $\Phi_A := \Phi \cup (\varphi_f)_{f\in F}$; finally, we set
  \[ \overline R := \langle \Phi_A \rangle_X.
  \]
  Then $\overline R|_A = S|_A$, as the following calculation shows:
  Let $x \in A$. By construction, $\varphi_f \in \llbracket S\rrbracket$
  for every~$f \in F$. In particular, $\overline R \subset S$ and thus
  $\overline R|_A \subset S|_A$. 
  Conversely, let $x, y \in A$ with~$x \sim_S y$. Then there exist~$f,g \in F$
  with~$y \in f(R \cdot g^{-1}(x))$, whence $f^{-1}(y) \sim_R g^{-1}(x)$. By construction, we thus
  have
  \[ y \sim_{\overline R} f^{-1}(y)
  \quad\land\quad f^{-1}(y) \sim_{\overline R} g^{-1}(x)
  \quad\land\quad g^{-1}(x) \sim_{\overline R} x,
  \]
  and so~$y \sim_{\overline R} x$. 
  This shows $\overline R|_A = S|_A$. 

  In combination with Lemma~\ref{lem:restrict} we obtain
  \begin{align*}
    \cost_{\mu|_A}(S|_A)
    & = \cost_{\mu|_A}(\overline R|_A)
    \leq \cost_{\mu}(\overline R)
    \\
    & \leq \cost_{\mu}(\Phi_A)
    \leq \cost_{\mu} (\Phi) + |F| \cdot \mu(A)
    \\
    & \leq \cost_{\mu}(R) + \varepsilon + |F| \cdot \varepsilon.
  \end{align*}
  Taking~$\varepsilon \rightarrow 0$ shows that
  \[ \cost_\mu(S) \leq \cost_{\mu|_A}(S|_A) + 1 \leq \cost_{\mu}(R) + 1,
  \]
  as claimed.
\end{proof}

\section{Integral foliated simplicial volume}\label{sec:ifsv}

Simplicial volumes are defined as the minimal number (measured in a suitable sense)
of singular simplices needed to build the given manifold~\cite{vbc,mapsimvol}. 
In the case of integral foliated simplicial volume, we use bounded
functions on dynamical systems of the fundamental group as coefficients. More
information and computations can be found in the
literature~\cite{mschmidt,loehpagliantini,FLPS,fauser,fauserfriedlloeh,braun}.

Let $M$ be an oriented closed connected $n$-manifold with fundamental
group~$\Gamma$ and let $\alpha = \Gamma \actson (X,\mu)$ be a standard
$\Gamma$-space. Then $\linfz{(X,\mu)}$ inherits a right $\Z\Gamma$-module
structure and we write 
\[ C_*(M;\alpha) := \linfz {(X,\mu)} \otimes_{\Z \Gamma} C_*(\widetilde M;\Z)
\]
for the corresponding chain complex with twisted coefficients.

A chain~$c \in C_*(M;\alpha)$ is an \emph{$\alpha$-parametrised
  fundamental cycle} if it is homologous (in the
complex~$C_*(M;\alpha)$) to the image of a $\Z$-fundamental cycle
on~$M$ under the canonical inclusion~$C_*(M;\Z) \longrightarrow
C_*(M;\alpha)$.  If $c = \sum_{j=1}^m f_j \otimes \sigma_j \in
C_n(M;\alpha)$ is in reduced form (i.e., all $\sigma_1,\dots,
\sigma_m$ lie in different $\Gamma$-orbits under the deck
transformation action), then
\[ |c|_1 := \sum_{j=1}^m \int_X |f_j| \;d\mu \in \R_{\geq 0}.
\] 

\begin{defi}[parametrised simplicial volume, integral foliated simplicial volume]
  Let $M$ be an oriented closed connected $n$-manifold.
  \begin{itemize}
  \item The \emph{$\alpha$-parametrised simplicial volume} of~$M$ is defined
    by
    \begin{align*}
      \ifsvp M \alpha
      := \inf \bigl\{ |c|_1 \bigm|
      \;& \text{$c \in C_n(M;\alpha)$ is an $\alpha$-parametrised}
      \\
      & \text{fundamental cycle of~$M$}\bigr\}.
    \end{align*}
  \item The \emph{integral foliated simplicial volume}~$\ifsv M$ of~$M$
    is the infimum of all parametrised simplicial volumes of~$M$.
  \end{itemize}
\end{defi}

If $M$ is an oriented closed connected manifold, then classical simplicial volume,
integral foliated simplicial volume, and stable integral simplicial volume are related
by the chain~\cite[Theorem~5.35]{mschmidt}\cite[Proposition~6.1]{loehpagliantini}
\[ \sv M \leq \ifsv M \leq \stisv M.
\]

\begin{exa}
  Let $M$ be an oriented closed connected manifold with residually
  finite fundamental group~$\Gamma$. Then $\ifsvp M {\pfc 
    \Gamma}$ coincides with the stable integral simplicial volume
  of~$M$~\cite[Remark~6.7]{loehpagliantini}.
\end{exa}

\section{Proof of Theorem~\ref{mainthm}}\label{sec:proof}

We will now prove Theorem~\ref{mainthm}.

\subsection{The case of finite fundamental group}\label{subsec:finpi1}

Let us first get the (pathological) case that the fundamental
group~$\Gamma$ is finite out of the way: If $\alpha = \Gamma \actson
(X,\mu)$ is a standard $\Gamma$-space,
then~\cite[Corollaire~I.10]{gaboriaucost}\cite[Proposition~22.1]{kechrismiller}
\[ \cost_\mu \alpha  - 1 \leq 1 - 1 = 0 \leq \ifsvp M \alpha. 
\]
Taking the infimum over all such~$\alpha$ shows that $\cost \Gamma - 1
\leq \ifsv M$. This proves Theorem~\ref{mainthm} if $\Gamma$ is finite. 

\subsection{Setup}\label{subsec:setup}

In view of Section~\ref{subsec:finpi1}, we will assume for the rest of the
proof of Theorem~\ref{mainthm} that the fundamental group~$\Gamma$ of~$M$
is infinite. Moreover, we will fix the following notation:
\begin{itemize}
\item Let $\alpha = \Gamma \actson (X,\mu)$ be an essentially free ergodic
  standard $\Gamma$-space and let $S \subset X \times X$ be the corresponding
  orbit relation. Hence, $S$ is aperiodic and ergodic. 
\item Let $D \subset \widetilde M$ be a set-theoretic, relatively
  compact, fundamental domain for the deck transformation action
  of~$\Gamma$ on~$\widetilde M$.
\item Let
  \[ c = \sum_{j=1}^m f_j \otimes \sigma_j \in C_n(M;\alpha) =
  \linfz X \otimes_{\Z \Gamma} C_n(\widetilde M;\Z)
  \]
  with $f_1,\dots,
  f_m \in \linfz X$, $\sigma_1,\dots, \sigma_m \in
  \map(\Delta^n,\widetilde M)$ be an $\alpha$-parametrised fundamental
  cycle of~$M$. Moreover, we may assume that the representation of~$c$
  is in \emph{reduced form}, i.e., that $\sigma_j(v_0) \in D$ for all~$j \in \{1,\dots,m\}$
  and that the singular simplices~$\sigma_1,\dots, \sigma_m$ are all different. 
\item
  For~$j \in \{1,\dots, m\}$ let $\gamma_j \in \Gamma$ be the unique group element
  satisfying~$\sigma_j(v_1) \in \gamma_j \cdot D$. We 
  then consider~$\varphi_j \in \llbracket S \rrbracket$ given by 
  \begin{align*}
    \varphi_j \colon A_j & \longrightarrow \gamma_j^{-1} \cdot A_j
    \\
    x & \longmapsto \gamma_j^{-1} \cdot x,
  \end{align*}
  where $A_j := \supp f_j \subset X$. 
  Let $R := \langle \varphi_1,\dots, \varphi_m \rangle_X$ be the standard equivalence
  relation on~$X$ generated by~$\varphi_1,\dots, \varphi_m$.
\end{itemize}

By construction, the cost of the relation~$R$ is controlled in terms of~$|c|_1$:

\begin{lem}\label{lem:costnorm}
  In the situation of Setup~\ref{subsec:setup}, the relation~$R$
  is a subrelation of~$S$ and
  \[ \cost_\mu R \leq \sum_{j=1}^m \mu(A_j) \leq |c|_1.
  \]
\end{lem}
\begin{proof}
  By construction, $\Phi := (\varphi_j)_{j \in \{1,\dots, m\}}$ is a
  graphing of~$R$ and we have $\varphi_1, \dots, \varphi_m \in \llbracket S\rrbracket$.
  Therefore, $R \subset S$ and 
  \[ \cost_\mu R \leq \cost_\mu \Phi
     = \sum_{j=1}^m \mu(A_j).
  \]
  Moreover, $\sum_{j=1}^m \mu(A_j) \leq \sum_{j=1}^m \int_X |f_j| d\mu = |c|_1$,
  because each~$f_j$ is integer-valued and the representation~$\sum_{j=1}^m f_j \otimes\sigma_j$
  of~$c$ is in reduced form.
\end{proof}

\begin{rem}
  Integration~$\linfz X \otimes_{\Z \Gamma} C_*(\widetilde M;\Z)
  \longrightarrow C_*(M;\R)$ of the coefficients and a covering
  theoretic argument show that $\langle \gamma_1, \dots,
  \gamma_m\rangle_\Gamma$ is a finite index subgroup of~$\Gamma$.
  However, in general, the subrelation~$R$ of~$S$ will \emph{not} have
  finite index in~$S$ (this can already be seen in the case of
  Schmidt's parametrised fundamental cycles
  of~$S^1$~\cite[proof of Proposition~5.30]{mschmidt}).
\end{rem}

In view of Lemma~\ref{lem:costnorm} it suffices to prove that
$\cost_\mu S - 1 \leq \cost_\mu R$. To this end, we will establish
that $R \subset S$ is a translation finite extension and then apply
Lemma~\ref{lem:transfin}.

\subsection{Passing to locally finite chains}

In order to prove that $R\subset S$ is a translation finite extension,
it is convenient to pass to locally finite chains.

\begin{rem}[locally finite cycles]\label{rem:lf}
  In the situation of Setup~\ref{subsec:setup}, for $\mu$-almost
  every~$x \in X$, the chain
  \[ c_x := \sum_{j = 1}^m \sum_{\gamma \in \Gamma} f_j(\gamma^{-1} \cdot x) \cdot \gamma \cdot \sigma_j
  \in \Clf n {\widetilde M} \Z
  \]
  given by evaluation on the $\Gamma$-orbit of~$x$ 
  is a well-defined locally finite fundamental cycle of~$\widetilde
  M$~\cite[Lemma~2.5]{FLPS}. 
\end{rem}

We therefore recall a basic property of locally finite chains.

\begin{lem}\label{lem:lfres}
  Let $N$ be an oriented connected $n$-manifold, let $Z$ be a
  commutative ring with unit, and let $x \in N$. Then the restriction
  map induces a well-defined isomorphism
  \begin{align*}
    \varrho_x \colon \Hlf n N Z
    & \longrightarrow H_n(N, N\setminus\{x\};Z)
    \\
    \biggl[ \sum_{j \in J} a_j \cdot \sigma_j
      \biggr]
    & \longmapsto
    \biggl[ \sum_{j \in J, x \in \sigma_j(\Delta^n)} a_j \cdot \sigma_j
      \biggr].
  \end{align*}
  In particular: If $c = \sum_{j \in J} a_j \cdot \sigma_j \in \Clf n N Z$ is a locally finite
  cycle whose associated class~$[c] \in \Hlf n N Z$
  is non-trivial, then there exists~$j \in J$ such that
  \[ x \in \sigma_j(\Delta^n).
  \]
\end{lem}
\begin{proof}
  The restriction map on the chain level extends to a well-defined
  chain map~$\Clf * N Z \longrightarrow C_*(N, N\setminus
  \{x\};Z)$. Checking the effect of~$\varrho_x$ on the locally finite
  fundamental class of~$N$ proves the first claim.
  
  The second part is a direct consequence of the first part.
\end{proof}

\subsection{Establishing translation finiteness}

\begin{lem}\label{lem:geom}
  In the situation of Setup~\ref{subsec:setup}, $R \subset S$ is a
  translation finite extension in the sense of
  Definition~\ref{def:transfin}.
\end{lem}
\begin{proof}
  For the proof we will use geometric properties of~$c$ (and its locally finite
  companions) on~$\widetilde M$. 
  Let
  \[ K := D \cup \bigcup_{j=1}^m \sigma_j(\Delta^n) \subset \widetilde M.
  \]
  Then $D$ and $K$ are relatively compact and hence 
  \[ F := \{ f \in \Gamma \mid D \cap f \cdot K \neq \emptyset \}
  \]
  is finite. We will now show that $\mu$-almost every $S$-orbit is
  covered by the $F$-translates of orbits of~$R$:

  Let $x \in X$ be such that the evaluation~$c_x$ is a locally finite
  $\Z$-fundamental cycle of~$\widetilde M$ (Remark~\ref{rem:lf}).   
  We associate the following graph~$G_x = (V_x,E_x)$
  with~$c_x$: 
  \begin{itemize}
  \item vertices: we set
    $V_x := \bigl\{ \gamma \in \Gamma
       \bigm| \exi{j \in \{1,\dots, m\}} \gamma^{-1} \cdot x \in A_j
       \bigr\}
       \subset \Gamma.
    $
  \item edges: we set
    \begin{align*}
      E_x := \bigl\{ \{\gamma,\lambda\}
      \bigm| & \; \gamma, \lambda \in V_x \land \gamma \neq \lambda
      \\
      \land & \; \exi{i,j \in \{1,\dots, m\}} \exi{k,\ell \in \{0,\dots, n\}}
      \\
      & \qquad 
      \bigl(\ \partial_k( \gamma \cdot \sigma_i) = \partial_\ell(\lambda \cdot \sigma_j)
      \\
      &\qquad 
      \land \gamma^{-1} \cdot x \in A_i \land \lambda^{-1} \cdot x \in A_j
      \bigr) 
        \bigr\}.
    \end{align*}
  \end{itemize}
  The combinatorics of~$G_x$ will allow us to link the orbits of~$R$ 
  with the geometry of~$c_x$. More precisely, we will establish the following
  facts:
  \begin{enumerate}
  \item\label{item-edges} For all~$\{\gamma,\lambda\} \in E_x$, we have $\gamma^{-1} \cdot
    x \sim_R \lambda^{-1} \cdot x$. 
  \item\label{item-cycle} If $V \subset V_x$ is (the set of vertices of)
    a connected component of~$G_x$, then
    \[ c_{x,V} := \sum_{j=1}^m \sum_{\gamma \in V} f_j(\gamma^{-1}\cdot x) \cdot \gamma \cdot \sigma_j 
    \]
    is a well-defined cycle in~$\Clf n {\widetilde M} \Z$.
  \item\label{item-select}
    Let $x_0 \in D$.  There exists a connected
    component~$V \subset V_x$ of~$G_x$ such that $[c_{x,V}] \neq 0$ in~$\Hlf n
    {\widetilde M} \Z$ and
    \[ \exi{g \in V} \exi{j\in \{1,\dots,m\}}
       x_0 \in g \cdot \sigma_j(\Delta^n) \subset g \cdot K.
    \]
    Hence, any such $g$ is in~$F$.
  \item \label{item-FV}
    We have $F \cdot V^{-1} = \Gamma$, where
    $V^{-1} := \{ \gamma^{-1} \mid \gamma \in V \}$.
  \end{enumerate}

  \emph{Proof of~(\ref{item-edges}).}
  By definition of~$E_x$, there are~$i,j \in \{1,\dots,m\}$ and~$k,\ell \in \{0,\dots,n\}$
  with
  \[ \gamma^{-1} \cdot x \in A_i
  \ \land\  \lambda^{-1} \cdot x \in A_j
  \ \land\  \partial_k(\gamma\cdot \sigma_i) = \partial_\ell(\lambda \cdot \sigma_j).
  \]
  We now distinguish the following cases:
  \begin{itemize}
  \item If $k >0 $ and $\ell>0$, then
    \[ \gamma\cdot \sigma_i(v_0)
       = \bigl(\partial_k(\gamma \cdot \sigma_i)\bigr) (v_0)
       = \bigl(\partial_\ell(\lambda \cdot \sigma_j)\bigr) (v_0)
       = \lambda \cdot \sigma_j(v_0).
    \]
    In particular, $\gamma \cdot D \cap \lambda \cdot D \neq \emptyset$, and
    so~$\gamma = \lambda$ (whence~$\gamma^{-1} \cdot x \sim_R \lambda^{-1} \cdot x$). 
  \item If $k = 0$ and $\ell > 0$, then
    \[ \gamma \cdot \sigma_i(v_1)
       = \bigl(\partial_k(\gamma \cdot \sigma_i)\bigr) (v_0)
       = \bigl(\partial_\ell(\lambda \cdot \sigma_j)\bigr) (v_0)
       = \lambda \cdot \sigma_j(v_0).
    \]
    By definition of~$\gamma_i$, we have~$\sigma_i(v_1) \in \gamma_i
    \cdot D$. Therefore, we obtain that $\gamma \cdot \gamma_i \cdot D \cap
    \lambda \cdot D \neq \emptyset$, and thus~$\gamma \cdot \gamma_i =
    \lambda$. Because of~$\gamma^{-1} \cdot x \in A_i$, the definition
    of~$R$ shows that
    \[ \lambda^{-1} \cdot x = \gamma_i^{-1} \cdot
    \gamma^{-1} \cdot x \sim_R \gamma^{-1} \cdot x.
    \]
  \item If $k >0 $ and $\ell = 0$, we can argue as in the previous case.
  \item If $k = 0$ and $\ell = 0$, then
    \[ \gamma \cdot \sigma_i(v_1)
       = \bigl(\partial_k(\gamma \cdot \sigma_i)\bigr) (v_0)
       = \bigl(\partial_\ell(\lambda \cdot \sigma_j)\bigr) (v_0)
       = \lambda \cdot \sigma_j(v_1).
    \]
    Similarly, to the previous cases, we obtain~$\gamma \cdot \gamma_i = \lambda \cdot \gamma_j$. Hence
    \[ \lambda^{-1} \cdot x \sim_R \gamma_j^{-1} \cdot \lambda^{-1} \cdot x
    = \gamma_i^{-1} \cdot \gamma^{-1} \cdot x
    \]
    (via~$\varphi_j$) and
    \[ \gamma_i^{-1} \cdot \gamma^{-1} \cdot x \sim_R \gamma^{-1} \cdot x
    \]
    (via~$\varphi_i$). By transitivity, it follows that~$\lambda^{-1} \cdot x \sim_R \gamma^{-1} \cdot x$.
  \end{itemize}
  
  \emph{Proof of~(\ref{item-cycle}).}
  Let $\pi_0(G_x)$ be the set (of vertex sets) of the connected
  components of~$G_x$.  The sum decomposition
  $c_x = \sum_{V \in \pi_0(G_x)} c_{x,V}
  $
  is a locally finite sum of locally finite chains. Hence,  
  \begin{align*}
    0 & = \partial (c_x) = \sum_{V \in \pi_0(G_x)} \partial (c_{x,V})
    \\
    & =
    \sum_{V \in \pi_0(G_x)} \sum_{j=1}^m \sum_{\gamma \in V} \sum_{k=0}^n (-1)^k
    \cdot f_j(\gamma^{-1} \cdot x) \cdot \partial_k(\gamma \cdot \sigma_j)
    .  
  \end{align*}
  By construction of the graph~$G_x$, if $V,W \in \pi_0(G_x)$ are different
  components, then the terms of~$\partial(c_{x,V})$ and $\partial (c_{x,W})$
  cannot interfere with each other. Therefore, we obtain
  \[ \partial(c_{x,V}) = 0
  \]
  for all~$V \in \pi_0(G_x)$.
  
  \emph{Proof of~(\ref{item-select}).}  By~(\ref{item-cycle}), $c_x =
  \sum_{V \in \pi_0(G_x)} c_{x,V} $ is a locally finite sum of cycles.
  Applying the restriction homomorphism~$\varrho_{x_0}$ of
  Lemma~\ref{lem:lfres} gives the effectively finite decomposition
  \[ 0 \neq [\widetilde M, \widetilde M \setminus \{x_0\}]_\Z
     = \varrho_{x_0} [\widetilde M]^{\lf}_\Z = \varrho_{x_0}[c_x]
     = \sum_{V \in \pi_0(G_x)} \varrho_{x_0} [c_{x,V}].
  \]
  Hence, there exists a connected component~$V \in \pi_0(G_x)$
  with~$[c_{x,V}] \neq 0$ in~$\Hlf n {\widetilde M} \Z$. 
  By Lemma~\ref{lem:lfres}, there exist~$g \in V$ and $j \in \{1,\dots,m \}$
  with
  \[ x_0 \in g \cdot \sigma_j(\Delta^n).
  \]
  By definition of~$F$ and because~$x_0 \in D$, this implies~$g \in F$.
  
  \emph{Proof of~(\ref{item-FV}).}
  Clearly, $F\cdot V^{-1} \subset \Gamma$. Conversely, let $\gamma \in \Gamma$.
  Applying Lemma~\ref{lem:lfres} to the point~$\gamma^{-1} \cdot x_0$ and the class~$[c_{x,V}]
  \in \Hlf n {\widetilde M} \Z$ 
  yields that there exists a~$\lambda \in V$ and~$j \in \{1,\dots, m\}$ with
  \[ \gamma^{-1} \cdot x_0 \in \lambda \cdot \sigma_j(\Delta^n).
  \]
  Thus, $\gamma^{-1} \cdot D \cap \lambda \cdot K \neq \emptyset$, and so
  $\gamma \cdot \lambda \in F
  $. Hence, $\gamma \in F \cdot \lambda^{-1} \subset F \cdot V^{-1}$. 

  \emph{Conclusion of proof}:
  Let $V \subset V_x$ and $g \in V \cap F$ be as provided by fact~(\ref{item-select}).
  Then (\ref{item-edges}) shows that
  \[ V^{-1} \cdot x\subset R \cdot g^{-1} \cdot x.
  \]
  Using~(\ref{item-FV}), we obtain
  that
  \[ S \cdot x = \Gamma \cdot x
  = \bigcup_{f \in F} f \cdot V^{-1} \cdot x
  \subset \bigcup_{f,g \in F} f \cdot R \cdot g^{-1}(x)
  \subset S \cdot x.
  \]
  Because translation by~$f \in F$ lies in~$[S]$, this shows that $R \subset S$
  is a translation finite extension.
\end{proof}

\subsection{Putting it all together}\label{subsec:final}

We continue to use the setup from Section~\ref{subsec:setup}. 
Because $R \subset S$ is a translation finite extension (Lemma~\ref{lem:geom}),
we obtain
\[ \cost_\mu S \leq \cost_\mu R + 1
\]
from Lemma~\ref{lem:transfin}. In combination with Lemma~\ref{lem:costnorm},
it follows that
\begin{align*}
  \cost_\mu \alpha
  & = \cost_\mu S \leq \cost_\mu R + 1
  \\
  & \leq \sum_{j=1}^m \mu(A_j) + 1\leq |c|_1 + 1.
\end{align*}
Taking the infimum over all $\alpha$-parametrised fundamental cycles~$c$ of~$M$
thus shows the desired estimate
\[ \cost_\mu \alpha - 1 \leq \ifsvp M \alpha.
\]

Because integral foliated simplicial volume can be computed in terms
of ergodic essentially free parameter
spaces~\cite[Proposition~4.17]{loehpagliantini}, taking the infimum
over all ergodic essentially free standard $\Gamma$-spaces~$\alpha$
implies that
\[ \cost \Gamma -1 \leq \ifsv M.
\]
This completes the proof of Theorem~\ref{mainthm}.

\subsection{The weightless version}\label{subsec:weightless}

The proof of the cost estimate of Theorem~\ref{mainthm} does not
incorporate the values of the coefficient functions. Therefore, the
estimate can be improved in a straightforward way to the case of
weightless parametrised simplicial volumes
(Theorem~\ref{mainthmwl}). The advantage of these weightless versions
is that they also allow for coefficients in finite fields and other
commutative rings with unit~\cite{loehfp}.

We quickly review the definition of weightless parametrised simplicial
volumes and indicate how to prove the theorem in this case. Let $M$
be an oriented closed connected $n$-manifold with fundamental group~$\Gamma$,
let $\alpha = \Gamma \actson (X,\mu)$ be a standard $\Gamma$-space, and
let $Z$ be a commutative ring with unit.
We then write
$L^\infty(X,Z) := Z \otimes_\Z  \linfz X
$ 
and
\[ C_*(M;\alpha;Z) := L^\infty(X,Z) \otimes_{\Z \Gamma} C_*(\widetilde M;\Z).
\]
A cycle~$c \in C_*(M;\alpha;Z)$ is an $(\alpha;Z)$-fundamental cycle of~$M$
if $c$ is homologous to a $Z$-fundamental cycle of~$M$. Moreover, we define
the \emph{weightless norm} of a chain~$c = \sum_{j =1}^m f_j \otimes \sigma_j \in C_n(M;\alpha;Z)$
in reduced form by
\[ |c|_{(\alpha;Z)} := \sum_{j=1}^m \mu(\supp f_j) \in\R_{\geq 0}.
\]
The \emph{weightless parametrised $Z$-simplicial volume} of~$M$ is given by
\begin{align*}
  \suv M {\alpha;Z} := \inf \bigl\{ |c|_{(\alpha;Z)} \bigm|
  \;& \text{$c \in C_n(M;\alpha;Z)$ is an $(\alpha;Z)$-fundamental}
  \\
  & \text{cycle of~$M$}\bigr\}.
\end{align*}

\begin{thm}\label{mainthmwl}
  Let $M$ be an oriented closed connected manifold with fundamental
  group~$\Gamma$, let $\alpha = \Gamma \actson (X,\mu)$ be an
  essentially free ergodic standard $\Gamma$-space, and let $Z$ be a
  commutative ring with unit. Then
  \[ \cost_\mu \alpha -1 \leq \suv M {\alpha;Z}.
  \]
\end{thm}

It should be noted that as coefficient ring~$Z$ in
Theorem~\ref{mainthmwl} we can also take, e.g., finite fields. Therefore,
we obtain an upper bound of cost in terms of objects in positive characteristic. 

\begin{proof}
  We can prove this version in the same way as the $\ell^1$-version
  in Theorem~\ref{mainthm}. We will therefore only indicate the basic steps: 
  \begin{itemize}
  \item As in the $\ell^1$-case, we can assume without loss of generality that $\Gamma$
    is infinite. Let $S$ be the orbit relation of~$\alpha$.
  \item Let $c = \sum_{j=1}^m f_j \otimes \sigma_j
    \in L^\infty(X,Z) \otimes_{\Z \Gamma} C_n(\widetilde M;\Z)$ be
    an $(\alpha,Z)$-fun\-da\-men\-tal cycle of~$M$ in reduced form.
  \item
    Literally in the same way as in the $\ell^1$-case, we define the relation~$R$
    on~$X$ associated with~$c$.
  \item
    Then Lemma~\ref{lem:costnorm} shows that $R$ is a subrelation of~$S$ and
    \[ \cost_\mu R \leq |c|_{(\alpha;Z)}.
    \]
  \item For $\mu$-almost every~$x \in X$, the chain
    \[ c_x = \sum_{j=1}^m \sum_{\gamma \in \Gamma} f_j(\gamma^{-1} \cdot x) \cdot \gamma \cdot \sigma_j
    \]
    is a well-defined locally finite $Z$-fundamental cycle in~$\Clf n
    {\widetilde M} Z$ of~$\widetilde M$ (the proof of the
    $\Z$-case~\cite[Lemma~2.5]{FLPS} also works for $Z$-coefficients).
  \item Using Lemma~\ref{lem:lfres} and the arguments of the proof of
    Lemma~\ref{lem:geom}, we obtain that $R \subset S$ is a
    translation finite extension.
  \item As in Section~\ref{subsec:final}, we thus obtain~$\cost_\mu
    \alpha -1 \leq |c|_{(\alpha;Z)} \leq \suv M {\alpha;Z}$, as claimed.
    \qedhere
  \end{itemize}
\end{proof}


\medskip
\vfill

\noindent
\emph{Clara L\"oh}\\[.5em]
  {\small
  \begin{tabular}{@{\qquad}l}
    Fakult\"at f\"ur Mathematik,
    Universit\"at Regensburg,
    93040 Regensburg\\
    \textsf{clara.loeh@mathematik.uni-r.de},\\
    \textsf{http://www.mathematik.uni-r.de/loeh}
  \end{tabular}}


\begin{thebibliography}{100}
  
  \bibitem{abertnikolov}
    M.~Ab\'ert, N.~Nikolov.
    Rank gradient, cost of groups and the rank versus Heegard genus problem,
    \emph{J.~Eur.\ Math.\ Soc.},~14, 1657–1677, 2012.  

  \bibitem{braun}
    S.~Braun. 
    \emph{Simplicial Volume and Macroscopic Scalar Curvature}.
    PhD thesis, Karlsruhe Institute of Technology, 2018.
    
  \bibitem{fauser}
    D.~Fauser.
    Integral foliated simplicial volume and $S^1$-actions,
    preprint, available at \textsf{arXiv:1704.08538 [math.GT]}, 
    2017.

  \bibitem{fauserfriedlloeh}
    D.~Fauser, S.~Friedl, C.~L\"oh.
    Integral approximation of simplicial volume of graph manifolds,
    to appear in \emph{Bull.\ Lond.\ Math.\ Soc.},
    \textsf{DOI 10.1112/blms.12266} 

  \bibitem{feldmanmoore}
    J.~Feldman, C.C.~Moore.
    Ergodic equivalence relations, cohomology, and von Neumann algebras.~I. 
    \emph{Trans.\ Amer.\ Math.\ Soc.}, 234(2), 289--324, 1977.
    
  \bibitem{FFM}
     S.~Francaviglia, R.~Frigerio, B.~Martelli. 
     Stable complexity and simplicial volume of manifolds, 
     \emph{J.~Topol.}, 5(4), 977--1010, 2012.
    
  \bibitem{FLPS}
    R.~Frigerio, C.~L\"oh, C.~Pagliantini, R.~Sauer.
    Integral foliated simplicial volume of aspherical manifolds, 
    \emph{Israel J.\ Math.}, 216(2), 707--751, 2016. 

  \bibitem{gaboriaucost}
    D.~Gaboriau. Co\^{u}t des relations d'\'{e}quivalence et des groupes, 
    \emph{Invent.\ Math.}, 139(1), 41--98, 2000. 

  \bibitem{gaboriaul2}
    D.~Gaboriau.
    Invariants $\ell^2$ de relations d'\'equivalence et de groupes,
    \emph{Inst.\ Hautes \'Etudes Sci.\ Publ.\ Math.}, 95, 93--150, 2002. 
             
  \bibitem{vbc}
    M.~Gromov.
    Volume and bounded cohomology,
    \emph{Inst.\ Hautes \'Etudes Sci.\ Publ.\ Math.}, 56, 5--99, 1983.

  \bibitem{kechrismiller}
    A.S.~Kechris, B.D.~Miller.
    \emph{Topics in Orbit Equivalence},
    Springer Lecture Notes in Mathematics, vol.~1852, 2004.
    
  \bibitem{mapsimvol}
    C.~L\"oh.
    Simplicial volume,
    \emph{Bull.\ Man.\ Atl.}, 7--18, 2011.

  \bibitem{loehrg}
    C.~L\"oh.
    Rank gradient vs.\ stable integral simplicial volume, 
    \emph{Period.\ Math.\ Hung.}, 76(1), 88--94, 2018. 

  \bibitem{loehfp}
    C.~L\"oh.
    Simplicial volume with $\F_p$-coefficients,
    to appear in \emph{Period.\ Math.\ Hung.}, 2019.
    
  \bibitem{loehpagliantini}
    C.~L\"oh, C.~Pagliantini.
    Integral foliated simplicial volume of hyperbolic $3$-manifolds,
    \emph{Groups Geom.\ Dyn.}, 10(3), 825--865, 2016. 

  \bibitem{lueckapprox}
    W.~L\"uck. 
    Approximating $L^2$-invariants by their finite-dimensional analogues, 
   \emph{Geom.\ Funct.\ Anal.}, 4(4), 455--481, 1994. 
	
  \bibitem{mschmidt} 
    M.~Schmidt. 
    \emph{$L^2$-Betti Numbers of
    $\mathcal{R}$-spaces and the Integral Foliated Simplicial
    Volume}. PhD~thesis, Westf\"alische Wilhelms-Universit\"at
    M\"unster, 2005.\\ 
    \textsf{http://nbn-resolving.de/urn:nbn:de:hbz:6-05699458563}


\end{thebibliography}
\end{document}